\documentclass[12pt]{amsart}
\usepackage{geometry}
\usepackage{graphicx}
\usepackage{amssymb}
\usepackage{epstopdf}
\usepackage{amsmath,amscd}
\usepackage{amsthm}
\usepackage{enumerate}
\usepackage{url,verbatim,soul}
\usepackage{subfig}
\usepackage{enumitem}

\RequirePackage[colorlinks,citecolor=blue,urlcolor=blue]{hyperref}
\usepackage{breakurl}
\theoremstyle{plain}
\newtheorem{theorem}{Theorem}
\newtheorem{definition}[theorem]{Definition}
\newtheorem{lemma}[theorem]{Lemma}

\newtheorem{remark}[theorem]{Remark}

\newcommand\es{\varnothing}

\newcommand\Aut{\mathrm{Aut}}

\newcommand\sH{{\mathcal H}}

\newcommand\TT{{\mathbb T}}

\newcommand\RR{{\mathbb R}}
\newcommand\ZZ{{\mathbb Z}}

\newcommand\LL{{\mathbb L}}
\newcommand\PP{{\mathbb P}}

\renewcommand\a{\alpha}
\newcommand\om{\omega}
\newcommand\g{\gamma}

\newcommand\eps{\epsilon}
\newcommand\De{\Delta}
\newcommand\qq{\qquad}
\newcommand\q{\quad}
\newcommand\resp{respectively}

\newcommand\oo{\infty}
\newcommand\sG{{\mathcal G}}

\newcommand\Ga{\Gamma}

\newcommand\de{\delta}

\newcommand\id{{\bf 0}}

\newcommand\pd{\partial}

\newcommand\ghf{graph height function}
\newcommand\sghf{square \ghf}

\newcommand\hdi{$\sH$-difference-invariant}

\newcounter{mycount}\newcounter{mycount2}\newcounter{mycount3}
\newenvironment{romlist}{\begin{list}{\rm(\roman{mycount2})}%
   {\usecounter{mycount2}\labelwidth=1cm\itemsep 0pt}}{\end{list}}

\newenvironment{letlist}{\begin{list}{\rm(\alph{mycount})}%
   {\usecounter{mycount}\labelwidth=1cm\itemsep 0pt}}{\end{list}}

\newcounter{newcount1}

\newcommand\ga{\gamma}

\numberwithin{equation}{section}
\numberwithin{theorem}{section}
\numberwithin{figure}{section}

\title{On counting polygons in a crystal}

\author{Geoffrey R.\ Grimmett}
\address{Centre for
Mathematical Sciences, Cambridge University, Wilberforce Road,
Cambridge CB3 0WB, UK} 

\email{g.r.grimmett@statslab.cam.ac.uk}
\urladdr{\url{http://www.statslab.cam.ac.uk/~grg/}}

\date{9 December 2025} 

\begin{document}

\begin{abstract}
How many $n$-step polygons exist that contain a given vertex of an infinite quasi-transitive graph $G$?  
The exponential growth rate of such polygons is identified as the connective constant when $G$
has sub-exponential growth and possesses a so-called square graph height function. The last condition
amounts to the requirement that $G$ has a certain $\ZZ^2$ action of automorphisms.
The main theorem extends 
a result of Hammersley (Proc.\ Cambridge Philos.\ Soc.\ 57 (1961) 516--523) and others for the hypercubic lattice, and responds
to Hammersley's challenge to prove such a result for more general \lq\lq crystals''.
\end{abstract}
\maketitle

\section{Introduction}\label{sec:intro}

A \emph{self-avoiding walk} (SAW)
on a graph $G$ is a path that visits no vertex more than once.
The study of SAWs was initiated in the chemical theory of polymerisation
(see \cite{Orr} and the book \cite{F} of Flory). 
In their visionary paper \cite{HM}, Hammersley and Morton
studied \emph{inter alia} the number of $n$-step SAWs on a lattice $G$. 
They used subadditivity to prove the existence of
the exponential growth rate, that is, the limit
\begin{equation}\label{eq:connc}
\mu(G):= \lim_{n\to\oo} c_n^{1/n},
\end{equation}
where $c_n$ is the number of $n$-step SAWs on $G$ starting at a given vertex. 
This investigation was continued by Hammersley alone in \cite{Ham-cc}.
The constant $\mu(G)$ was termed the \emph{connective constant} of $G$ in \cite{BHam}, where
\eqref{eq:connc} is stated without proof.

In 1961, Hammersley \cite{Ham61} extended the theory from counts of SAWs to counts 
of (self-avoiding) polygons. He showed 
that, in the case of the hypercubic lattice $\ZZ^d$, the exponential growth rates of the numbers of 
$n$-step SAWs and polygons are equal. The purpose of the current note is to provide a response 
to Hammersley's challenge
to determine \lq\lq what properties a crystal must possess'' in order that this be the case.

In a further work \cite{HW62}, Hammersley and Welsh introduced a technique that
has proved very useful in studying SAWs and polygons, namely the combinatorics of so-called \lq bridges'.
In broad terms made specific in Section \ref{sec:not}, a bridge is a SAW with extremal endpoints.
The key property of bridges is that one may extend a bridge $b$ by appending a second bridge $b'$ to 
the final vertex of $b$. It is proved in \cite{HW62} that, for the hypercubic lattice $\ZZ^d$, the exponential
growth rate of the number of $n$-step bridges equals the connective constant $\mu(\ZZ^d)$.

In the current paper, we study polygons in quasi-transitive graphs $G$ satisfying certain conditions, thereby extending results of
\cite{Ham61} beyond the hypercubic lattices $\ZZ^d$. The main result of this paper is Theorem \ref{thm:main}, which we summarise as follows.

\begin{theorem}[Equality of growth rates]\label{thm:summ}
Let $G\in\sG$ have a \sghf\ and sub-exponential growth, and let 
$p_n$ be the number of $n$-step polygons that include the root of $G$. Then
$\pi(G) :=\limsup_{n\to \oo}p_n^{1/n}$ satisfies
$\pi(G) = \mu(G)$.
\end{theorem}

The set $\sG$ and the term \lq \sghf' are explained in Section \ref{sec:not}. Some illustrations of the terms
and conclusion of this theorem are provided in Sections \eqref{ex-amen}--\eqref{ex:tz},
where examples are provided where either of the conditions of the theorem fail and for which indeed $\pi<\mu$.
A discussion of the missing liminf is found after Theorem \ref{thm:main}.

The challenge in the current work lies in doing without some of the symmetries
of $\ZZ^d$ that underly earlier work on counts of polygons. 
This may be seen as a continuation of the work of Grimmett and Li that was
directed at understanding the properties of connective constants of general quasi-transitive graphs
(see \cite{GL19} for a review). 

Theorem \ref{thm:summ} provides a sufficient condition on the graph $G$  for the equality $\pi(G)=\mu(G)$
to hold.
It is believed that the strict inequality $\pi(G) < \mu(G)$ holds for non-amenable graphs, and some comments on
this inequality and its connection to the property of ballisticity are included in Sections \ref{ex-amen}--\ref{ssec:35}.

The reader is referred to \cite{BDGS,G-pog,MS} and \cite[Chap.\ 7]{Rens}
for general accounts of the combinatorics of SAWs and their cousins, and
to \cite{GL19} for SAWs on quasi-transitive graphs.

This paper is structured as follows.
Section \ref{sec:not} is devoted to background terminology and properties for graphs and their height functions.
The principal Theorem \ref{thm:main} follows in Section \ref{sec:main},  
together with some examples and a discussion of the ballisticity of random SAWs. The proof 
of Theorem \ref{thm:main} is found in Section \ref{sec:pf}.

\section{Preliminaries}\label{sec:not}

Throughout this paper, $G=(V,E)$ will denote an infinite, connected,  
locally finite, quasi-transitive graph with root denoted $\id$.
For simplicity, we  assume $G$ has neither parallel edges nor loops,
and we write $\sG$ for the set of all such rooted graphs.
If $u$ and $v$ are neighbours in $G$ (written $u\sim v$), we write $\langle u,v\rangle$ for the edge joining them.
The set of neighbours of $v$ is denoted $\pd v$. The graph-distance $d(u,v)$ is the 
number of edges in the shortest path from $u$ to $v$.

The automorphism group of $G=(V,E)$ is
denoted $\Aut(G)$. A subgroup $\Ga \le \Aut(G)$ is said to \emph{act
transitively} on $G$  (or on its vertex-set $V$)
if, for $v,w\in V$, there exists $\g \in \Ga$ with $\g (v)=w$.
The subgroup $\Ga$  is said to  \emph{act quasi-transitively} if there is a finite
set $W$ of vertices 
such that, for $v \in V$, there exist
$w \in W$ and $\g \in \Ga$ with $\g (v) =w$.
The graph $G$ is called \emph{transitive} 
(\resp, \emph{quasi-transitive}) if $\Aut(G)$ acts transitively
(\resp, quasi-transitively). 
The orbit of a vertex $v$ under $\Ga$ is denoted $\Ga v$. The number of orbits of $V$ under $\Ga$ is written
as $M(\Ga) =|V/\Ga|$.

A \emph{cycle} of $G$ is a sequence $(v_0,v_1,\dots,v_m)$ with $m\ge 3$
such that $v_i\sim v_{i+1}$ for $0\le i < m$,
$v_m=v_0$, and  $v_0,v_1,\dots,v_{m-1}$ are distinct vertices. The length of a cycle is the number of edges
traversed. There is some indecision in the literature concerning the terms cycle, circuit, polygon, 
and we shall define the last in the next paragraph.

A \emph{self-avoiding walk} (SAW) on $G$ is a path starting at $\id$ that visits no vertex more than once.
A \emph{polygon} is a cycle of $G$ containing the root $\id$; that is, a polygon, when oriented, comprises 
a SAW from $\id$ to some 
neighbour $v$ of $\id$, 
together with the edge $\langle v,\id\rangle$. A SAW (\resp, polygon) is said to have $n$ steps if it has exactly $n$ edges.
Let $c_n$ be the number of $n$-step SAWs from $\id$, and let $p_n$ be the number of $n$-step polygons
containing $\id$.
We shall also be interested in \lq bridges', which will be defined soon.

By the above definition of a polygon, we have that
\begin{equation}\label{eq:ineq1}
2p_n = c_{n-1}(\pd\id) \le c_n, \qq n\ge 3,
\end{equation}
where $c_m(\pd \id)$ denotes the number of $m$-step SAWs from $\id$
ending at some neighbour of $\id$.
We seek here conditions on $G$ for which the exponential growth rates of $c_n$ and $p_n$ are equal, and towards
 this end 
we introduce the concept of a \ghf.

\begin{definition} \cite[Defn 3.1]{GrL4}\label{def:height}
Let $G =(V,E)\in \sG$ with root labelled $\id$. 
A \emph{\ghf} (abbreviated to \lq ghf') on $G$ is a pair $(h,\sH)$ such that
\begin{letlist}
\item $h:V \to\ZZ$, and $h(\id)=0$, 
\item $\sH\le \Aut(G)$ acts quasi-transitively on $G$ 
such that $h$ is \emph{\hdi}, in that
$$
h(\a v) - h(\a u) = h(v) - h(u), \qq \a \in \sH,\ u,v \in V,
$$
\item for  $v\in V$,
there exist $u,w \in \pd v$ such that
$h(u) < h(v) < h(w)$.
\end{letlist}
\end{definition}

Note that, if $(h,\sH)$ is a ghf, then so is $(-h,\sH)$. Examples of graphs in $\sG$ possessing a ghf are found in \cite{GrL4},
and of graphs without a ghf in \cite{GL-amen}. While possession of a ghf allows progress on SAWs 
and their so-called locality problem for connective constants, 
one needs more for the study of polygons.

\begin{definition}\label{def:ht2}
Let $(h,\sH)$ be a \ghf\ for $G\in\sG$, and let $\rho\in\sH$. We call $(h,\sH,\rho)$ a \emph{square
\ghf} (abbreviated to \lq square ghf') if
\begin{letlist}
\item $\rho$ is a translation, in that it fixes no finite set $F$ of vertices,
\item $\rho$ is height-preserving in that, for $v\in V$, we have $h(v)=h(\rho(v))$,
\item $\rho$ commutes with every $\alpha\in\sH$.
\end{letlist}
\end{definition}

The term \lq square ghf' should not be confused with the \lq strong ghf' of \cite{GL-amen}, and it is
motivated as follows. Under the conditions of the definition,
for a non-height-preserving $\alpha\in\sH$, the action on $V$ of the pair $(\rho,\alpha)$ 
is as the square lattice $\ZZ^2$. One might relax somewhat these conditions, but 
for simplicity we retain the above.

These definitions are utilized as follows. Firstly, it was proved in \cite{GrL4} that, if $G$ has a ghf,
then one may define the notion of a \lq bridge' on $G$, and moreover the bridge  
growth rate $\beta$ equals the connective constant $\mu$.
This is elaborated later in this section. Secondly, our main theorem, Theorem \ref{thm:main},
asserts that a graph with a square ghf has the property that $\limsup p_n^{1/n}=\beta$ (which in turn equals $\mu$).

We discuss bridges next. 
Bridges were introduced (but not by that name) in \cite{HW62}
in the context of the hypercubic lattice $\ZZ^d$. Let $(h,\sH)$ be a ghf for $G\in\sG$.
An $h$-\emph{bridge} $(v_0,v_1,\dots,v_n)$
is a SAW on $G$ satisfying
$$
h(v_0)< h(v_m) \le h(v_n), \qq 0< m \le n.
$$
Let $b_n$ be the number of $n$-step $h$-bridges $\pi$ starting at $v_0=\id$.
Using quasi-transitivity and subadditivity, it may be shown
that the limit 
\begin{equation}\label{eq:beta}
\beta(h):=\lim_{n\to\oo} b_n^{1/n}
\end{equation}
exists, and $\beta(h)$ is called the $h$-\emph{bridge constant}. 

Note that the definition \eqref{eq:beta} of $\beta(h)$ depends on the choice of \ghf\ $h$.
It was noted above that, if $(h,\sH)$ is a ghf, then so is $(-h,\sH)$. It turns out that $\beta(h)=\beta(-h)$ if
$\sH$ is what is called \lq unimodular', and in this case the \emph{bridge constant} $\beta$ 
is defined to be their
common value. In the non-unimodular case, we define $\beta$ by $\beta=\max\{\beta(h), \beta(-h)\}$.
In each case we have that $\beta=\mu$, whence $\beta=\beta(G)$ is independent of the choice of ghf.
The above is summarised in the following theorem.

\begin{theorem}\label{thm:bridge}
Let $G\in \sG$ possess a ghf $(h,\sH)$, and denote by $\mu(G)$ the connective constant of $G$. 
\begin{letlist}
\item\cite{GrL4} If $\sH$ is unimodular, then $\beta(h)=\beta(-h)$, and their common value 
$\beta(G)$ satisfies $\beta(G)=\mu(G)$.
\item\cite{Lind} If $\sH$ is non-unimodular, then $\beta(G):=\max\{\beta(h),\beta(-h)\}$
satisfies  $\beta(G)=\mu(G)$.
\end{letlist}
\end{theorem}

For discussions of unimodularity see \cite[Sect.\ 8.2]{LyP} or \cite[Sect.\ 3]{GrL4}.

The \emph{growth function} of a transitive graph $G$ is defined as
\begin{equation*}
\Ga_n=|\{v\in V: d(0,v)\le n\}|,
\end{equation*}
that is, the number of vertices in a ball of radius $n$. The graph is said to have \emph{sub-exponential growth} if 
\begin{equation}\label{eq:sube}
\lim_{n\to\oo} \frac1n \log \Ga_n = 0. 
\end{equation}

\begin{remark}\label{rem:1}
It is useful to recall that any quasi-transitive graph $G$ with sub-exponential growth is amenable,  
and hence any automorphism group of such $G$ that acts quasi-transitively 
is unimodular (see \cite{SW} and \cite[Exer.\ 8.30]{LyP}).
\end{remark}

\section{Main theorem}\label{sec:main}

\subsection{Statement of main theorem}\label{ssec31}
For a rooted graph $G$, let
\begin{equation}\label{eq:defpi}
\pi=\pi(G) := \limsup_{n\to \oo}p_n^{1/n},
\end{equation}
where $p_n$ is the number of $n$-step polygons of $G$ containing the root $\id$. The three main characters of
this article are $\mu(G)$, $\beta(G)$, $\pi(G)$.

\begin{figure}[t]
\centerline{\raisebox{10pt}{\includegraphics[width=0.5\textwidth]{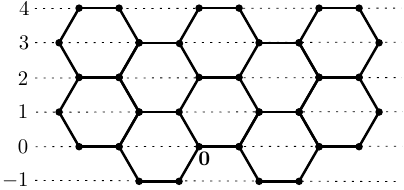}}
\q\includegraphics[width=0.36\textwidth]{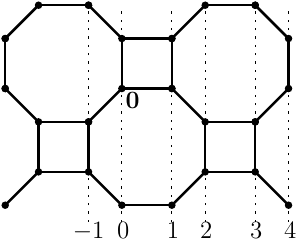}}
\caption{The hexagonal lattice and the square/octagon lattice. The heights of vertices are as marked. 
The automorphism $\rho$ is a suitable shift rightwards for the first, and a suitable shift upwards for the second.}\label{fig:latt}
\end{figure} 

\begin{remark}\label{rem:ind}
Recall that, for quasi-transitive graphs $G\in\sG$,
 $\mu(G)$ is independent of the choice of root (see \cite{Ham-cc}).
Furthermore, for such $G$ that in addition possess a ghf, $\beta(G)$ 
exists and is independent of the choice of root (see \cite{GrL4, Lind}
and Theorem \ref{thm:bridge}).
Subject to the conditions of the following theorem, $\pi(G)$ is also independent of the choice of root.
\end{remark}

\begin{theorem}\label{thm:main}
Let $G\in\sG$ have a square ghf $(h,\sH,\rho)$ and sub-exponential growth. 
\begin{letlist}
\item We have that 
$\pi(G) = \beta(G)=\mu(G)$.
\item Furthermore, for $\eps>0$, there exists an arithmetic sequence $(m_N: N\ge 1)$ of integers along which
$\liminf_{N\to\oo} p_{m_N}^{1/m_N} \ge \beta(G)-\eps$.
\end{letlist}
\end{theorem}

The limsup of $p_n^{1/n}$
is identified in Theorem  \ref{thm:main}, but the story of its liminf is incomplete. 
The missing element seems to be a proof of the convergence of $p_n^{1/n}$.
The proof of this for $\ZZ^d$ uses a fairly simple geometric construction together with subadditivity;
this is due to Hammersley \cite{Ham61}, but see also \cite[Thm 3.2.3]{MS}.
The geometry
is however more challenging in the generality of the current paper. 
There exist nevertheless graphs satisfying the conditions of Theorem 
\ref{thm:main} to which the subadditivity argument may be adapted, 
so long as a certain extra condition on the ghf holds. We do not investigate this here.

It is a tautology that a bipartite graph $G$ has no odd polygons. Thus, for bipartite graphs
(such as the hypercubic, hexagonal, and square/octagon lattices),
the liminf is to be taken along the even integers only.

By Theorem \ref{thm:main}, the exponential growth rates of polygon and SAW counts are equal (under the stated conditions). This is much weaker than proving  concrete polynomial bounds on $p_n/c_n$, as may be found in part for $\ZZ^d$ in
\cite[Thm 1.1]{DCGHM}.   

The importance of the two conditions of the theorem are illustrated by some examples.

\subsection{Euclidean lattices}\label{ex-amen}\cite[Sect.\ 3]{GrL4}
Many lattices embedded in $\RR^d$ satisfy the conditions of Theorem \ref{thm:main}. 
We mention two planar examples, namely the
hexagonal lattice and the square/octagon lattice as  illustrated in Figure \ref{fig:latt}.

\begin{figure}[t]
\centering
\includegraphics[width=0.5\textwidth]{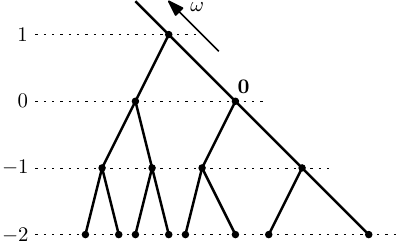}
\caption{The $3$-regular tree with the \lq horocyclic' height
function.}\label{fig:tree}
\end{figure}

\subsection{Tree graphs}\label{ex:tree}\cite[Sect.\ 3]{GrL4}
The $3$-regular tree $\TT_3$ has exponential growth, and possesses a ghf but no square ghf,
as follows. Let $\om$ be a ray of $\TT_3$, 
and `suspend' $\TT_3$ from $\om$ (as illustrated  in Figure \ref{fig:tree}). A given vertex
on $\omega$  is labelled $\id$ and has
height $h(\id)=0$, and other vertices have their horocyclic heights, that is, their generation numbers relative to $\id$.

Let $\sH$ be the set of automorphisms of $\TT_3$ that fix 
the end of $\TT_3$ determined by $\om$, noting that $\sH$ is non-unimodular (see \cite[Rem.\ 3.3]{GrL4}). 
Then $(h,\sH)$ is a ghf. It follows (and is in fact trivial) that $\beta(\TT_3)=\mu(\TT_3)=2$.
On the other hand, $\TT_3$ has no square ghf and, since it has no cycles we have $\pi(\TT_3)=0$. 

\begin{figure}[t]
 \centerline{\includegraphics[width=0.4\textwidth]{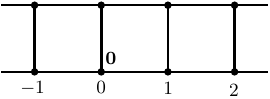}}
  \caption{The ladder graph $\LL_2$ with the root and heights indicated.}
  \label{fig:ladder}
\end{figure}

\subsection{Ladder graphs}\label{ex:ladder}\cite[Sect.\ 3]{GL-bnds}
The ladder graph $\LL_2:=\ZZ\times\{0,1\}$  of Figure \ref{fig:ladder} has sub-exponential growth,
and possesses a ghf (as indicated in the figure) but no square ghf. 
Furthermore, $\beta(\LL_2)=\mu(\LL_2)=\phi$ and $\pi(\LL_2)=1$,
where $\phi=\frac12(1+\sqrt 5) = 1.618\dots$ is the golden mean. 

Now consider the \lq triple ladder' $\LL_3:=\ZZ\times \{0,1,2\}$. By
Theorem \ref{thm:bridge}, $\beta(\LL_3)=\mu(\LL_3)$.
For simplicity, 
suppose a polygon $p$ goes rightwards from $\id$. It forms a bridge of $\LL_2$ (viewed as a subgraph
of $\LL_3$) until it reverses direction. 
Then it follows a leftwards bridge on a copy of $\LL_2$, before once again turning back towards the origin
(the shapes of the latter bridges are constrained by those of the earlier).
The number of such $n$-step polygons  is no larger than the number of $n$-step
bridges of $\LL_2$ (disregarding minor terms). Therefore, $\pi(\LL_3)\le \beta(\LL_2)$. 
However, $\beta(\LL_2) < \beta(\LL_3)=1.914\dots$,  (\cite[p.\ 198]{JA}).
This may be extended to wider ladders of the form $\LL_m:=\ZZ\times\{0,1,\dots,m\}$.

\subsection{Products of trees and lines}\label{ex:tz}\cite{GN,H19}
Let $G=\TT_3\times\ZZ$ be the (Cartesian) product of the $3$-regular tree $\TT_3$ and the doubly-infinite line $\ZZ$;
vertices of $G$ are expressed as vectors $(t,z)$, and two vertices $(t_1,z_1)$, $(t_2,z_2)$ are adjacent if and only if
either $t_1=t_2$ and $|z_1-z_2|=1$, or $t_1 \sim t_2$ in $\TT_2$ and $z_1=z_2$. 
With $(h,\sH)$ as in Section \ref{ex:tree}, we define the ghf $(h',\sH')$ on $G$ as follows.
For $\a\in\sH$, define $\a'$ by  $\a'(t,z):=(\a(t), z)$, and let $\rho$ be the shift given by $\rho(t,z)=(t,z+1)$.
The set of all such $\a'$, together with $\rho$, generates a (non-unimodular) subgroup $\sH'$ of $\Aut(G)$
that acts transitively. We let $h'(t,z) = h(t)$, thus obtaining the required ghf $(h',\sH')$ on $G$. It may be checked
that $(h',\sH',\rho)$ is a square ghf.

The graph $G$ has exponential growth and possesses a square ghf, whence $\beta(G)=\mu(G)$ by
Theorem \ref{thm:bridge}. By \cite[Thm 1.4]{H19} and the forthcoming Theorem \ref{thm:ball}, 
we have that $\pi(G)<\beta(G)$.
Such conclusions are valid  more generally for the graphs $\TT_k\times\ZZ^d$ (in the natural notation) with
$k\ge 3$ and $d\ge 1$.

One may try to prove $\pi<\beta$ using the Nachmias--Peres criterion $(\Delta-1)\rho<\mu$, where
$\De$ denotes vertex-degree and $\rho$ denotes spectral radius (see \cite[p.\ 6]{NP}).  This may be done
successfully for $k\ge 4$, $d=1$, using computed numerical inequalities for $\mu$. Hutchcroft \cite[p.\ 2804]{H19}
has shown that this criterion fails for $k=3$, $d=1$.

\subsection{Comments on main theorem}\label{ssec:35}
We make some comments on the proof of Theorem \ref{thm:main}. Proofs 
for the special case $\ZZ^d$ may be found in Hammersley \cite{Ham61}, Kesten \cite{KII}, Madras--Slade \cite{MS},
and Hughes \cite{H}.
The relevant properties of $\ZZ^d$ include reflection-invariance and translation-invariance, and these are used
in varying degrees in the four proofs. In addition, the proof of \cite{MS} seems to use the 
Euclidean geometry of $\RR^d$. 
Certain aspects of translation-invariance
are preserved in the generality of the current article via the assumption of the existence of a ghf.
The use of reflection-invariance is however more problematic in the case of more general
graphs. This is avoided here by an argument that is motivated by
Hammersley's original proof from 1961, but which avoids the gap in that proof at \cite[p.\ 520]{Ham61}
(see below). 

Whereas this paper is directed at the \emph{equality} of $\pi$ and $\mu$, 
earlier work of Madras--Wu \cite{MW} and Panagiotis
\cite{Pana} is devoted to \emph{strict inequality} (that is, $\pi<\mu$)
for regular tilings of the hyperbolic
plane. Such tilings have exponential growth. Panagiotis \cite{Pana} proved also 
that, for any transitive graph $G$ (and indeed more generally, see \cite[Rem.\ 4.1]{Pana}), 
such strict inequality implies that random SAW on $G$ is 
\emph{ballistic} (sometimes expressed as having \lq positive speed') in the sense that
there exists $c>0$ such that
\begin{equation}\label{eq:blaa}
\PP_n\bigl(d(\id,L_n)\le c n\bigr) \le e^{-cn}, \qq n \ge 1,
\end{equation}
where $\PP_n$ is the uniform probability measure on the set of $n$-step SAWs from the root $\id$ of $G$,
and $L_n$ is the final endvertex of the randomly selected SAW.
Hutchcroft \cite{H19} has proved ballisticity for any graph with a transitive, non-unimodular
(sub)group of automorphisms. The  strict inequality $\pi<\mu$ is believed to hold for all non-amenable, transitive graphs.

For clarity and reference, we present an explicit statement of the relationship 
between strict inequality and ballisticity (without claiming any originality).

\begin{theorem}\label{thm:ball}
Let $G\in\sG$ and let 
\begin{equation}\label{eq:reps}
\pi'(G)=\limsup_{n\to\oo}\Bigl(\sup_{v\in V} p_{n,v}\Bigr)^{1/n},
\end{equation}
where $p_{n,v}$ is the number of $n$-step polygons that include $v$.
Then $\pi'(G)<\mu(G)$ if and only if there exist $c>0$  and $N\ge 1$ such that
\begin{equation}\label{eq:blaa2}
\PP_{n,v}\bigl(d(v,L_n)\le c n\bigr) \le e^{-cn}, \qq n \ge N,\q v\in V,
\end{equation}
where $\PP_{n,v}$ is the uniform probability measure on the set of $n$-step SAWs starting at $v$,
and $L_n$ is the final endvertex of the randomly selected SAW.
\end{theorem}

Since $G$ is assumed quasi-transitive, the inner supremum of \eqref{eq:reps} is over a finite set.
If the conditions of Theorem \ref{thm:main} are satisfied, then $\pi'(G)=\pi(G)$.

\begin{proof}
That $\pi'<\mu$ implies \eqref{eq:blaa2} for suitable $c$, $N$ is a consequence of 
\cite[Thm 4, Rem.\ 4.1]{Pana}\footnote{Panagiotis's proof is similar to (but independent of)
a proof of Hammersley in \cite[p.\ 520]{Ham61}. Furthermore, it corrects an error in the latter proof, 
where it is stated  (in the language of that paper) that $\Phi(n)$ 
is the union of the $\Phi_\psi(n)$.}.
Conversely, by \eqref{eq:ineq1} and \eqref{eq:blaa2}, 
\begin{align*}
p_{n+1,v } &= \tfrac12 c_{n,v} \PP_{n,v}\bigl(d(v,L_n)=1\bigr) \le \tfrac12 c_{n,v} e^{-c n},  
\qq n \ge N,\q v\in V,
\end{align*} 
where $c_{n,v}$ is the number of $n$-step SAWs from $v$.
On recalling Remark \ref{rem:ind}, we deduce that $\pi'(G) \le \mu(G)e^{-c}$ as required.
\end{proof}

By Theorems \ref{thm:main} and \ref{thm:ball}, random SAWs on the graphs $G$ of
Theorem \ref{thm:main} are not ballistic. That is of course weaker than
showing they are sub-ballistic in the sense of \cite{DCH,KP}.

\section{Proof of Theorem \ref{thm:main}}\label{sec:pf}

We begin with some further notation.
Let $(h,\sH,\rho)$ be a square ghf on $G=(V,E)\in\sG$.
A SAW from $u$ to $v$ is called \emph{stiff} if all of its vertices $x$, other than
its endvertices,  satisfy $h(u)<h(x)< h(v)$.
We shall define  (as in \cite{GrL4}) a certain  integer $r=r(h,\sH)$. If $\sH$ acts transitively, we set
$r=0$. Assume $\sH$ does not act transitively, and 
let $r=r(\sH)$ be the infimum of all $r$ such that the following holds.
Let $o_1,o_2,\dots,o_M$ be representatives of the (finitely many) orbits of $\sH$.
For $i\ne j$, there
exists a vertex $v_j \in \sH o_j$ 
together with a stiff SAW $\nu(o_i,v_j)$ from $o_i$
to $v_j$ with length $r$ or less. 
We choose such a SAW, and denote it $\nu(o_i,v_j)$ as above. We set $\nu(o_i,o_i)=\{o_i\}$.

It is proved in \cite[Prop.\ 3.2]{GrL4} that
\begin{equation}\label{eq:r}
0\le r \le (M-1)(2d+1)+2,
\end{equation}
where $M=|V/\sH|$ and 
\begin{equation}\label{eq:ddef}
d= \max\bigl\{|h(x)-h(y): x,y\in V,\, x\sim y\bigr\}.
\end{equation}
The constant $r$ will be used later in this section. 
By \eqref{eq:ddef},
\begin{equation}\label{eq:dh}
d(x,y) \ge \frac 1d |h(x)-h(y)|, \qq x,y\in V. 
\end{equation}

We state a lemma next.

\begin{lemma}\label{lem:new1}
Let $G\in\sG$ have square ghf $(h,\sH,\rho)$.
\begin{letlist}
\item We have that, as $k\to\infty$, $d(v,\rho^k(v)) \to \oo$ uniformly in $v\in V$.
\item
There exists an integer $\de$ such that
\begin{equation*}
d(v,\rho^k(v)) \le k \de, \qq v\in V,\q k \ge 1.
\end{equation*}
\end{letlist}
\end{lemma}

\begin{proof}
(a) Suppose there exists $v$ such that 
$d(v,\rho^k(v))\not\to\oo$. Since $G$ is locally finite, there exists $w\in V$ and a subsequence 
$(k_i)$ such that $\rho^{k_i}(v)=w$ for all $i$. Thus $\rho$ fixes the set of vertices $\{\rho^k(v): j_1\le k < j_2$\},
in contradiction of the assumption that $\rho$ is a translation. 

Next we prove uniformity of divergence. Firstly, we claim that, for given $k$, $d(v,\rho^k(v))$ is constant
on orbits of $\sH$. To see this, 
let $u\in \sH v$ and find $\alpha\in\sH$ such that $v=\alpha(u)$. Since $\rho\alpha=\alpha\rho$,
\begin{equation}\label{eq:new5}
d(v,\rho^k(v)) = d\bigl(\alpha(u), \rho^k(\alpha(u))\bigr)
=d\bigl(\alpha(u), \alpha(\rho^k(u))\bigr)
= d(u, \rho^k(u)).
\end{equation}

Secondly, suppose $u\notin\sH v$. 
By \eqref{eq:r},  there exists a path of length not
exceeding $r$ from $u$ to some $w\in \sH v$. 
By the triangle inequality, 
\begin{align}\label{eq:new6}
d(u,\rho^k(u)) &\le d(u,w) + d(w,\rho^k(w)) + d(\rho^k(w),\rho^k(u))\\
&\le 2r + d(w,\rho^k(w))
= 2r + d(v,\rho^k(v))\qq\text{by \eqref{eq:new5}}.
\nonumber
\end{align}
The claimed uniformity follows from \eqref{eq:new5}--\eqref{eq:new6}.

(b)
By \eqref{eq:new5}--\eqref{eq:new6} with $k=1$, $d(v,\rho(v))$ is bounded above by some $\de>0$, uniformly in $v$. 
By the triangle inequality,
\begin{equation*}
d(v,\rho^k(v)) \le \sum_{l=0}^{k-1} d(\rho^l(v), \rho^{l+1}(v)),
\end{equation*}
and the inequality for general $k$ follows.
\end{proof}

Let $G$ satisfy the conditions of the theorem with square ghf $(h,\sH, \rho)$.
We write $\pi:=\limsup_{n\to \oo}p_n^{1/n}$ as in \eqref{eq:defpi}. 
By \eqref{eq:ineq1} and Theorem \ref{thm:bridge},
\begin{equation}\label{eq:6}
\pi \le \mu=\beta.
\end{equation}

Since $G$ has sub-exponential growth, it is unimodular (see Remark \ref{rem:1}).
By Theorem \ref{thm:bridge}(a), the bridge constant $\beta=\beta(G)$ satisfies $\beta=\beta(h)$. 

\begin{figure}[t]
 \centerline{\includegraphics[width=0.4\textwidth]{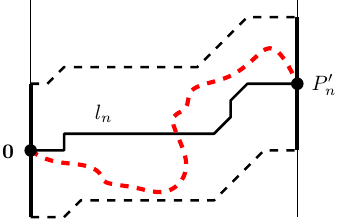}}
  \caption{The region $R_n$. A shortest path $l_n$ joins $\id$ to $P_n'$, and the red (dashed) path is a bridge.
  The two vertical lines depict the sets of $v$ such that $h(v)=0$, and such that $h(v)=h(P_n')$, \resp.}
  \label{fig:box}
\end{figure}

Let $n\ge 1$ and recall the number $b_n$ of $h$-bridges (henceforth called simply bridges). 
Find a vertex $P_n$ with $h(P_n)\ge 1$ such that the number of
bridges from $\id$ to $P_n$ is at least $b_n/|\Ga_n|$. Instead of working with the point $P_n$, we work with a point 
$P_n'$ defined as follows. By the discussion leading to \eqref{eq:r}, 
there exists a vertex $P_n' \in \sH\id$ such that there is a stiff SAW $\nu$ from $P_n$ to $P_n'$
with some length $\ell$ not exceeding $r$, We adjoin $\nu$ to any bridge from $\id$ to $P_n$ to obtain that
\begin{equation}\label{eq:exist}
\text{ there exist at least $b_n/|\Ga_n|$ bridges of length $n+\ell$ from $\id$ to $P_n'$.}
\end{equation}
We assume henceforth that $n\ge \ell$.

Let $l_n$ be a shortest bridge from $\id$ to $P_n'$, and let
$R_n$ be the subgraph of $G$ induced by the vertex-set
\begin{equation*}
D_n:= \{\id\}\cup\bigl\{v\in V: d(v,l_n)\le n,\, 1\le h(v) \le h(P_n')\bigr\}.
\end{equation*}
Note that the bridges of \eqref{eq:exist} lie within $R_n$.

The region $R_n$ is the basic ingredient of the following construction. Let $\ga\in\sH$ be such that
$\ga(\id)=P_n'$, so that $\ga$ maps $R_n$ to an image $\ga(R_n)$
with heights between $h(P_n')$ and $2h(P_n')$.
We may think of the $\ga^i(R_n)$, $i\in\ZZ$, as consecutive translates of $R_n$ that fit together to form
a \lq tube' along which their bridges combine to create a longer bridge.
See Figures \ref{fig:box} and \ref{fig:poly}.

Let $N\ge 1$ and let $S_N= \bigcup_{0\le i\le N-1} \ga^i(R_n)$ and $P_n^N=\ga^{N-1}(P_n')$. By \eqref{eq:exist}, 
\begin{equation}\label{eq:10}
\text{there  exist at least $(b_n/|\Ga_n|)^N$
$(n+\ell)N$-step bridges in $S_N$ from $\id$ to $P_n^N$.}   
\end{equation}

Let $l_{n,\oo}$
be the union $\bigcup_{-\oo < i < \oo}\ga^i(l_n)$, considered as a doubly-infinite path.
We now extend the $S_N$ into a doubly-infinite tube containing $l_{n,\oo}$, 
by defining $S_\oo$ to be the subgraph 
of $G$ induced by the vertex-set
\begin{equation*}
D_{n,\oo}=\bigl\{v\in V: d(v,l_{n,\oo})\le n\bigr\}.
\end{equation*}
The graph $S_\oo$ is periodic in that $\ga(S_\oo)=S_\oo$. Furthermore, $S_N$ is a subgraph of $S_\oo$.

\begin{figure}[t]
 \centerline{\includegraphics[width=0.8\textwidth]{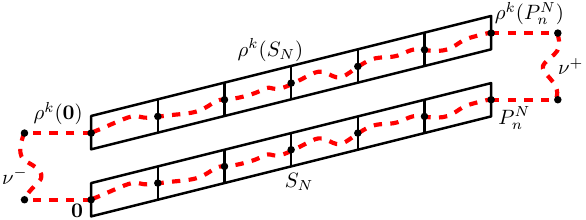}}
  \caption{The tube $S_N$ and its image $\rho^k(S_N)$.  Each
  is traversed by an $(n+\ell)N$-step  bridge, 
  and these two bridges are joined into a red (dashed) polygon by adding the connecting paths $\nu^\pm$.}
  \label{fig:poly}
\end{figure}

\begin{lemma}\label{prop:es}
There exists $k=k_n<\oo$ such that $S_\oo\cap\rho^k(S_\oo)  = \es$.
\end{lemma} 

\begin{proof}
Let $k \ge 1$, and suppose there exists $v\in S_\oo\cap \rho^k(S_\oo)$. 
Since $v\in \rho^k(S_\oo)$, we have $\rho^{-k}(v)\in S_\oo$,
so that $v,\rho^{-k}(v) \in S_\oo$. Thus there exist two vertices of the form $u, \rho^k(u)$ lying in $S_\oo$. 
We claim that, by Lemma \ref{lem:new1}(a),  this cannot hold for sufficiently large $k$.
To this end it suffices to show that 
\begin{equation*}
K:=\sup\{d(v,w): v,w\in S_\oo,\ h(v)=h(w)\}
\end{equation*}
satisfies
\begin{equation}\label{eq:Kfin}
K<\oo.
\end{equation}

We may express $l_{n,\oo}$ as the path $(\dots, z_{-1},z_0=\id, z_1,\dots)$ where 
$z_{-1}\in \ga^{-1}(R_n)$ and
$z_1\in R_n$. Since $h(z_k)\to \pm\oo$ as $k\to\pm\oo$, we have by \eqref{eq:dh} that
\begin{equation}\label{eq:ffin}
\text{for $v\in V$, $d(v, z_k) \to \oo$ as $k\to \pm\oo$.}
\end{equation}

By \eqref{eq:ffin} and the periodicity of $S_\oo$, there exists a finite subpath $\l'$ of $l_{n,\oo}$ such that
\begin{equation*}
K \le \sup\{d(v,w): v,w\in V,\, d(v, l')\le n,\ d(w,l')\le n\}.
\end{equation*} 
This is a supremum over a finite set, whence $K<\oo$, and the conclusion of the lemma follows 
by Lemma \ref{lem:new1}(a) for sufficiently large $k$.
\end{proof}

By Lemma \ref{prop:es}, we may choose $k=k_n<\oo$ such that $S_\oo\cap\rho^k(S_\oo)   = \es$.
Since $S_N \subseteq S_\oo$, 
\begin{equation}\label{eq:nooverlap}
S_N\cap \rho^k(S_N)=\es, \qq N\ge 1.
\end{equation}
By \eqref{eq:10}, there exist at least $(b_n/|\Ga_n|)^N$ distinct bridges traversing each of $S_N$ and $\rho^k(S_N)$, 
and we propose to join such bridges into  polygons by adding connections between their endvertices, as 
illustrated in Figure \ref{fig:poly}.

Consider first a connection between $\id$ and $\rho^k(\id)$ using vertices with negative height
(apart from its endvertices).
We construct three paths as follows.
\begin{romlist}
\item 
By Definition \ref{def:height}(c), we may find a path $(\id, c_1,c_2,\dots)$ such that
$h(c_i)\le -i$ for all $i$. By \eqref{eq:dh}, we may choose $t=t_n$ such 
that $d(H_0,c_{t})> k\de$ where $H_0=\{v\in V: h(v)=0\}$ and $\de$ is given in Lemma \ref{lem:new1}(b).
We denote by $\nu_1$ the path $(\id, c_1,c_2,\dots,c_t)$.
\item
By Lemma \ref{lem:new1}(b), there exists a shortest path $\nu_2$ from $c_t$
 to $\rho^k(c_t)$ with length not exceeding $k\de$.
\item
As above, the path $\nu_3:=\rho^k(\nu_1)$, reversed, joins $\rho^k(c_t)$ to $\rho^k(\id)$.
\end{romlist}
The union of the $\nu_i$ contains a path $\nu^-$
from $\id$ to $\rho^k(\id)$ whose vertices, apart from the two endvertices, have strictly negative heights.
The length of $\nu^-$ does not exceed $2t_n+ k\de$. 

By a similar construction we find a path $\zeta$ from $P_n^1$ to $\rho^k(P_n^1)$ using only vertices with 
heights strictly exceeding $h(P_n^1)$, except for its endvertices, and we set $\nu^+=\ga^{N-1}(\zeta)$.
Thus $\nu^+$ joins $P_n^N$ to $\ga^{N-1}(\rho(P_n^1)) = \rho(P_n^N)$, using only vertices
(apart from the endvertices) with heights strictly exceeding $h(P_n^N)$. 
Let $l_n^\pm$ be the lengths of the above $\nu^\pm$, so that
\begin{equation}\label{eq:llength}
l_n^\pm \le 2t_n+k\de.
\end{equation}

The above paths and bridges may be pieced together to form polygons. First we follow a bridge traversing
$S_N$ from $\id$ to
$P_n^N$, followed by $\nu_n^+$, followed by a bridge of $\rho^k(S_N)$ (reversed)
from $\rho^k(\id)$ to $\rho^k(P_n^N)$, and finally the path $\nu^-$ (reversed). 
It follows that
\begin{equation}\label{eq:5}
p_{2(n+\ell)N+l_n^-+l_n^+} \ge \left( \frac{b_n}{|\Ga_n|}\right)^{2N}.
\end{equation}
Take the $(2Nn)$th root to find that
\begin{equation}\label{eq:7}
 p_{2(n+\ell)N+l_n^-+l_n^+}^{1/(2Nn)} \ge \left(\frac{b_n}{|\Ga_n|}\right)^{1/n}.
\end{equation}
Now, by \eqref{eq:llength},
\begin{equation}\label{eq:8}
\begin{aligned}
\frac {2(n+\ell)N+l_n^-+l_n^+}{2nN} &\to 1+ \frac \ell n &&\text{as $N\to\oo$}\\
 &\to 1 &&\text{as $n\to\oo$},
 \end{aligned}
 \end{equation}
 whence $\pi \ge \beta$ by \eqref{eq:beta} and \eqref{eq:sube}.
This may be combined with \eqref{eq:6} to obtain part (a) of the theorem.

We turn to the more specific part (b). Let $\eps>0$ and choose $\eps',\eta>0$ such that
$(\beta-\eps')^{1-\eta} > \beta-\eps$. Pick $n$ sufficiently large that 
\begin{equation*}
\frac\ell n <\eta, \qq \left(\frac{b_n}{|\Ga_n|}\right)^{1/n}> \beta-\eps',
\end{equation*}
and consider the arithmetic sequence
$(m_N: N=1,2,\dots)$ where $m_N=2(n+\ell)N + l_n^-+l_n^+$. By \eqref{eq:7}--\eqref{eq:8},
\begin{equation*}
\liminf_{N\to\oo} p_{m_N}^{1/m_N} \ge (\beta-\eps')^{ 1-\eta} > \beta-\eps,
\end{equation*}
as required.

\section*{Acknowledgement}

The author thanks Tony Guttmann, Stuart Whittington, and Barry Hughes for their recollections of Hammersley's
second proof that $\pi(\ZZ^d)=\mu(\ZZ^d)$ when $d \ge 2$. That proof was published in Hughes \cite[p.\ 442]{H}.
Tom Hutchcroft has kindly commented on the inequality $\pi<\mu$ for graphs with a transitive,
non-unimodular subgroup
of automorphisms. 

\providecommand{\bysame}{\leavevmode\hbox to3em{\hrulefill}\thinspace}
\providecommand{\MR}{\relax\ifhmode\unskip\space\fi MR }
\providecommand{\MRhref}[2]{%
  \href{http://www.ams.org/mathscinet-getitem?mr=#1}{#2}
}
\providecommand{\href}[2]{#2}

\end{document}